\documentclass[letterpaper, 10 pt, conference]{ieeeconf}  

\IEEEoverridecommandlockouts                              
\usepackage{graphicx}
\usepackage{bm}
\newtheorem{thm}{Theorem}[section]

\newtheorem{lem}{Lemma}[section]

\newtheorem{defn}{Definition}[section]

\newtheorem{rem}{Remark}[section]

\usepackage{amsfonts}
 \usepackage{diagbox}
\usepackage{makecell}
\usepackage{enumerate}
\usepackage{amsmath}
\usepackage{mathrsfs}
\usepackage{amssymb}
\usepackage{graphicx}
\usepackage{subfigure}
\usepackage{epstopdf}
\usepackage{color}
\usepackage{float}
\usepackage{caption}
\usepackage{pifont}
\usepackage{graphicx,xcolor,float}
\usepackage{cases}
\usepackage{verbatim}
\usepackage{tabularx}
\usepackage{todonotes}
\usepackage{booktabs}
 \usepackage{threeparttable}
\usepackage{cite}
\usepackage{multirow}
\DeclareMathOperator{\col}{col}
\DeclareMathOperator{\sign}{sign}

\title{\LARGE \bf
An Adaptive Collaborative Neurodynamic Approach to Compute Nash Equilibrium in Normal-Form Games
}

\author{Jianing Chen
\thanks{Jianing Chen is with the Department of Mathematics, Harbin Institute of Technology, Weihai, 264209,  China, and also with the Department of Systems Engineering, City University of Hong Kong, Hong Kong {\tt\small jchen2265-c@my.cityu.edu.hk}}
}

\begin{document}

\maketitle
\thispagestyle{empty}
\pagestyle{empty}

\begin{abstract}
The Nash Equilibrium (NE), one of the elegant and fundamental concepts in game theory, plays a crucial part within various fields, including engineering and computer science. However, efficiently computing an NE in normal-form games remains a significant challenge, particularly for large-scale problems. In contrast to widely applied simplicial and homotopy methods, this paper designs a novel Adaptive Collaborative Neurodynamic Approach (ACNA), which for the first time guarantees both exact and global NE computation for general $N$-player normal-form games with mixed strategies, where the payoff functions are non-convex and the pseudo-gradient is non-monotone. Additionally, leveraging the adaptive penalty method, the ACNA ensures its state enters the constraint set in finite time, which avoids the second-order sufficiency conditions required by Lagrangian methods, and the computationally complicated penalty parameter estimation needed by exact penalty methods. Furthermore, by incorporating the particle swarm algorithm, it is demonstrated that the ACNA achieves global convergence to an exact NE with probability one. At last, a simulation is conducted to validate the effectiveness of the proposed approach.
\end{abstract}

\section{Introduction}
Game theory, which analyzes strategic interactions among rational players, witnesses expansive applications in various research areas, such as economics \cite{bierman1998game}, control theory \cite{amit2024games}, evolutionary biology \cite{ritzberger1995evolutionary}, and, more recently, computer science \cite{li2024survey}. Since John Nash's pioneering work \cite{nash1950equilibrium}, the NE has emerged as a fundamental solution concept, ensuring that the optimization conditions of each player are satisfied, given the strategies of all other players. Following its establishment and the proof of its existence for any finite games using Brouwer's fixed-point theorem, numerous researchers have developed various numerical approaches to compute an NE in normal-form games \cite{lemke1964equilibrium,rosenmuller1971generalization,garcia1973simplicial,herings2001differentiable,cao2025efficient}.

The complexity of computing an NE varies depending on the number of players. For two-player game, the condition of NE can be formulated as a linear complementarity system, making the Lemke-Howson algorithm \cite{lemke1964equilibrium} the earliest and still the most significant method to seek an NE. However, when there are three or more players, the problem becomes considerably more challenging, as it is no longer a linear problem. In such cases, its payoff functions are polynomial, leading to a non-convex structure with a non-monotone pseudo-gradient. Consequently, most gradient descent-based approaches lose their effectiveness \cite{he2022distributed,yuan2024event,Rao2023,Ye2020,li2024logical,huang2024distributed,deng2022nash,Wu2020,Zhou2005,Bianchi2021,deng2018distributed,meng2022attack,deng2021distributed,franci2021training,Wang2021}.

In recent decade, researches have made significant progress in designing approaches to compute an NE in $N$-player normal-form games. Notable examples include the generalized Lemke-Howson algorithm proposed by Rosenm\"{u}ller \cite{rosenmuller1971generalization} and Wilson \cite{wilson1971computing}, the simplicial approximation by Lemke and L\"{u}thi \cite{garcia1973simplicial}, as well as various homotopy techniques \cite{herings2001differentiable,cao2025efficient}. However, for games involving three or more players, almost all of the existing approaches can only yield approximate solutions \cite{rosenmuller1971generalization,garcia1973simplicial,herings2001differentiable,cao2025efficient}. As highlighted in a recent review \cite{li2024survey}, as far as we know, there is still no known approach capable of computing an exact NE for general $r$-player games with $r\geq 3$.

Fortunately, in the realm of infinite (or continuous) games, several effective approaches have been developed to ensure global convergence despite non-convex payoff functions \cite{xia2023collaborative,chen2024approaching}. For instance, a collaborative neurodynamic approach (CNA) was formulated to find the generalized Nash equilibrium in multi-cluster games with non-convex payoff functions \cite{xia2023collaborative}. Additionally, a conjugate-based ordinary differential equation was designed to achieve global NE seeking for non-convex games with a canonical form \cite{chen2024approaching}. However, in the case of normal-form games with $N$ players, the payoff functions do not satisfy the canonical form nor the second-order optimality conditions required in \cite{xia2023collaborative}. As a result, directly computing an NE for normal-form games remains challenging. One promising method involves reformulating the normal-form game as an optimization problem, providing new insights into this issue. A formulation was proposed in \cite{chatterjee2009optimization}, which transformed the normal-form game into an optimization problem with polynomial cost functions and constraints. Nevertheless, developing efficient approaches to solve such optimization problem remains difficult. Another method involves converting the normal-form game into a multi-objective constrained optimization problem, where a cone-$\epsilon$ MOEA algorithm was designed to compute the Pareto-optimal NE. However, this method loses efficiency when applied to general NE computation. Given these challenges, we pose and explore the following question: \textit{Can we construct an optimization problem whose global optimal solution corresponds to an NE of a normal-form game and develop an efficient approach to guarantee both global and exact NE computation?}

The primary contributions of this paper are concluded bellow
\begin{enumerate}
	\item In contrast to the simplicial methods, homotopy methods and gradient descent-based approaches \cite{lemke1964equilibrium,rosenmuller1971generalization,garcia1973simplicial,herings2001differentiable,cao2025efficient,he2022distributed,yuan2024event,Rao2023,Ye2020,li2024logical,huang2024distributed,deng2022nash,Wu2020,Zhou2005,Bianchi2021,deng2018distributed,meng2022attack,deng2021distributed,franci2021training,Wang2021}, this paper introduces, for the first time, a novel ACNA that ensures both global and exact NE computation in general $N$-player normal-form games. By reformulating the normal-form game as an optimization problem, it is proved that an NE of the game corresponds to a global minima of the optimization problem. Furthermore, the cost function of the optimization problem is continuous, differentiable, and nonnegative, while the constraint set is convex and bounded, which greatly reduces the computational complexity.
	\item Compared to the Lagrangian method in \cite{xia2023collaborative}, the adaptive penalty method proposed in this paper circumvents the need for second-order sufficiency condition of the cost function, which is not satisfied by this optimization problem. Additionally, we do not need the cost function to be a canonical form as induced by \cite{chen2024approaching}. Leveraging the adaptive penalty method, the state of the ACNA will enter the constraint set in finite time—an advantage not guaranteed by most Lagrangian-based approaches \cite{xia2005recurrent,xia2023collaborative,liu2016collective}. Moreover, in contrast to the exact penalty method \cite{yan2016collective}, the penalty parameters in the ACNA are autonomously determined based on the intrinsic characteristic of the constraints.
\end{enumerate}

\section{Preliminary and Problem Statement}
\subsection{Notations}
$\mathcal{R}$, $\mathcal{R}^m$ and $\mathcal{R}^{m\times n}$ represent the real number set, $m$-dimensional vector set and $m\times n$ matrix set, respectively. $\mathcal{R}^m_{+}$ denotes the $m$-dimensional positive vector set. $(\cdot)^{\top}$ and $\|\cdot\|$ denote the transpose and the Euclidean norm, respectively. For $x_i\in\mathcal{R}^m$ and $i\in\mathbb{V}=\{1,\cdots,N\}$, let $\col(x_1,\cdots,x_N)=[x_1^{\top},\cdots,x_N^{\top}]^{\top}$. $\mathbf{0}_m=\col(0,0,\cdots,0)\in\mathcal{R}^m$ and $\mathbf{1}_m=\col(1,1,\cdots,1)\in\mathcal{R}^m$.
$I_m$ is the $m\times m$ identity matrix. Denote diag$\{\varrho_{1},\cdots,\varrho_{N}\}$ as a diagonal matrix with principal diagonal elements $\varrho_{i}$ for $i\in\mathbb{V}$. Furthermore, $\times$ and $\otimes$ are the Cartesian product and the Kronecker product, respectively. For a group of sets $S_1,\cdots,S_N$, let $\Pi_{i\in\mathbb{V}}S_i=S_1\times\cdots\times S_N.$ Denote $\sign(x)$ as the signum function of a variable $x\in\mathbb{R}$, where $\sign(x)=1$ if $x>0$; $\sign(x)=-1$ if $x<0$; $\sign(x)=0$, otherwise.
\par
\subsection{Normal-Form Game with Mixed Strategies}
In this paper, we aim to compute the NE for finite normal-form game with mixed strategies. The following necessary notation is presented for further analysis \cite{cao2025efficient}. Consider a finite $N$-player normal-form game. Player $i \in \mathbb{V}$ has $m_i$ pure strategies. Denote $m=\sum_{i \in \mathbb{V}} m_i$. Player $i$'s pure strategy set is given as $S^i=\left\{s_j^i \mid j \in \mathbb{M}_i\right\}$ with $\mathbb{M}_i=\left\{1,2, \cdots, m_i\right\}$. The set of all pure strategy profiles is denoted as $S=\prod_{i \in \mathbb{V}} S^i$. A pure strategy profile $s=\col\left(s_{j_1}^1, s_{j_2}^2, \cdots, s_{j_N}^N\right) \in S$ can be usually rewritten as $s=\col\left(s^i, s^{-i}\right)$ with $s^{-i}=\left(s_{j_1}^1, \cdots, s_{j_{i-1}}^{i-1}, s_{j_{i+1}}^{i+1}, \cdots, s_{j_N}^N\right) \in S^{-i}=\prod_{h \in \mathbb{V} \backslash\{i\}} S^h$. For $i \in \mathbb{V}$, let $u^i: S \rightarrow \mathcal{R}$ represent the payoff function of player $i$. Player $i$'s mixed strategy is a probability distribution on $S^i$ represented as $x^i=\col\left(x_1^i, x_2^i, \cdots, x_{m_i}^i\right)$. Let $X^i$ be the set of player $i$ 's mixed strategies with $X^i=\left\{x^i=\col\left(x_1^i, x_2^i, \cdots, x_{m_i}^i\right) \in \mathcal{R}_{+}^{m_i} \mid \sum_{j \in \mathbb{M}_i} x_j^i=1\right\}$. For $x^i \in X^i,~ x_j^i$ is the probability assigned to pure strategy $s_j^i \in S^i$. The set of all mixed strategy profiles is given as $X=\prod_{i \in \mathbb{V}} X^i$. A mixed strategy profile $x=\col\left(x^1, x^2, \cdots, x^N\right) \in X \quad$ can be usually rewritten as $x=\col\left(x^i, x^{-i}\right)$ with $x^{-i}=\left(x^1, \cdots, x^{i-1}, x^{i+1}, \cdots, x^N\right) \in X^{-i}=\prod_{k \in \mathbb{V} \backslash\{i\}} X^k$. If $x \in X$ is played, then the probability that a pure strategy profile $s=\col\left(s_{j_1}^1, s_{j_2}^2, \cdots, s_{j_N}^N\right) \in S$ occurs is equal to $\prod_{i \in \mathbb{V}} x_{j_i}^i$. For $x \in X$, the expected payoff of player $i$ is thus given by $u^i(x)=\sum_{j \in \mathbb{M}_i} x_j^i u^i\left(s_j^i, x^{-i}\right)$ with $u^i\left(s_j^i, x^{-i}\right)=\sum_{s^{-i} \in S^{-i}} u^i\left(s_j^i, s^{-i}\right) \prod_{r\in\mathbb{V} \backslash\{i\}} x_{j_r}^r$. Given these notations, a finite $N$-player normal-form game is written as $\Gamma=\left\langle \mathbb{V}, X,\left\{u^i\right\}_{i \in \mathbb{V}}\right\rangle$.
\begin{defn}\cite{nash1950equilibrium}\label{dingyi1}
    A mixed strategy profile $x^* \in X$ is an NE if $u^i\left(x^*\right) \geq u^i\left(x^i, x^{*-i}\right)$ for all $x^i \in X^i$ and $i \in \mathbb{V}$.
\end{defn} 

In summary, for $i\in\mathbb{V}$, the $i$th player's task is to solve the following interdependent optimization problem
\begin{equation}\label{wenti1}
\begin{aligned}
\max _{x^i} & ~~u^i(x^i,x^{-i}) \\
\text { s.t. } & ~\sum_{j \in \mathbb{M}_i} x_j^i=1,~x_j^i\geq0,~j \in \mathbb{M}_i.
\end{aligned}
\end{equation}
\begin{rem}
By viewing the mixed strategy profile $x$ as a continuous variable, the discrete game $\Gamma$ with finite strategies is transformed into a continuous game \eqref{wenti1} with infinite strategies. It should be noted that the payoff function $u^i$ is multilinear, whose pseudo-gradient $PG(x)=\col(\partial_{x^1}u^1,\cdots,\partial_{x^N}u^N)$ is non-monotone when $N>2$. However, almost all of the NE seeking approaches relay heavily on the monotone condition for global and exact NE computation \cite{he2022distributed,yuan2024event,Rao2023,Ye2020,li2024logical,huang2024distributed,deng2022nash,Wu2020,Zhou2005,Bianchi2021,deng2018distributed,meng2022attack,deng2021distributed,franci2021training,Wang2021}.  As illustrated in \cite{li2024survey}, it is still an open problem for exact NE computation in general $r$-player games with $r\geq 3$.
\end{rem}

Although it is hard to directly solve the non-monotone game \eqref{wenti1}, there are still many techniques, such as the CNA, for non-convex optimization problems. Thus, in the following subsection, the non-monotone game \eqref{wenti1} will be properly transformed as an optimization problem.
\subsection{Optimization Problem Transformation}
In this subsection, the non-monotone game \eqref{wenti1} will be transformed as an optimization problem to lower computational difficulty.

As indicated by Definition \ref{dingyi1}, when $x^*$ is an NE, the inequality $u^i(x^*)\geq u^i(s^i_j,x^{*-i})$ holds for any $i\in\mathbb{V}$ and $j\in\mathbb{M}_i$. Thus, we can define 
\begin{equation}
\begin{aligned}
z_j^i(x)&=u^i\left(s_j^i, x^{-i}\right)-u^i(x),\\
Q_j^i(x)&=\max \left\{z_j^i(x), 0\right\}.
\end{aligned}
\end{equation}
Then, the non-monotone game \eqref{wenti1} can be transformed as the following optimization problem
\begin{equation}\label{wenti2}
\begin{aligned}
\min_x & ~~\tilde{Q}(x)=\sum_{i \in \mathbb{V}} \sum_{j\in\mathbb{M}_i}\left[Q_j^i(x)\right]^2 \\
\text { s.t. } & ~~x^i \in X^i,~ i \in \mathbb{V}.
\end{aligned}
\end{equation}

It is obvious that the cost function $\tilde{Q}$ is continuous, differentiable, and satisfies the inequality $\tilde{Q}(x) \geq 0,~\forall x \in X$. 
\begin{lem}
    $x^*$ is an NE of game \eqref{wenti1}, if and only if, it is a global minima of problem \eqref{wenti2}.
\end{lem}
\begin{proof}
 If $x^*$ is an NE of game \eqref{wenti1}, we can get $u^i\left(x^*\right) \geq u^i\left(x^i, x^{*-i}\right)$ for all $x^i \in X^i$ and $i \in \mathbb{V}$. Then, letting $x^i_j=1$ for any $j\in\mathbb{M}_i$, we have $u^i(x^*)\geq u^i(s^i_j,x^{*-i})$. Thus, $x^*$ is a global minima of optimization problem \eqref{wenti2}.

 On the contrary, if $x^*$ is a global minima of optimization problem \eqref{wenti2}, one derives $u^i(x^*)\geq u^i(s^i_j,x^{*-i})$ for any $i\in\mathbb{V}$, $j\in\mathbb{M}_i$. It gives that $u^i(x^*)\geq\sum_{j \in \mathbb{M}_i} x_j^i u^i\left(s_j^i, x^{-i}\right)$ when $x^i \in X^i$. Therefore, $x^*$ is an NE of game \eqref{wenti1}.
\end{proof}

Thus, we can compute an NE for the non-monotone game \eqref{wenti1} by seeking a global minima of problem \eqref{wenti2}. For the convenience of subsequent analysis, problem \eqref{wenti2} can be rewritten as 
\begin{equation}\label{wenti3}
\begin{aligned}
\min_x & ~~\tilde{Q}(x) \\
\text { s.t. } & ~~x\in\Omega\cap\Psi,
\end{aligned}
\end{equation}
in which $\Omega=\{x:g^i_{1j}(x^i_j)=-x^i_j\leq0,~g^i_{2j}(x^i_j)=x^i_j-1\leq0,~i\in\mathbb{V},~j\in\mathbb{M}_i\}$ and $\Psi=\{x:h(x)=\col(\sum_{j \in \mathbb{M}_1} x_j^1-1,\cdots,\sum_{j \in \mathbb{M}_N} x_j^N-1)=\mathbf{0}_N\}$. To ensure the boundedness of $\Omega$, we add the trivial constraints $x^i_j\leq1$ for any $i\in\mathbb{V},~j\in\mathbb{M}_i$.
\section{Main Results}
\subsection{Design of Adaptive Neurodynamic Approach}
In this subsection, an adaptive neurodynamic approach (ANA) for the optimization problem \eqref{wenti2} is proposed as follows: 
\begin{align}\label{gs1}
\dot{x}(t)\in-\xi\Big(G\big(x(t)\big)\Big)\nabla_{x}\tilde{Q}(x)-R(x(t)),
\end{align}
where $R(x(t))=\zeta(x(t))\Big(\partial G(x(t))+\zeta(x(t))\partial H(x(t))\Big)$ makes sure that players' strategies will enter the constraints, in which $\zeta(x(t))=\int_{0}^{t}\sign{\varepsilon(x(s))}\mathrm{d}s$, $\varepsilon(x(t))$ $=G(x(t))+H(x(t))$, $G(x)=\sum_{i=1}^{N}\sum_{j=1}^{m_i}(\max\{0,g^i_{1j}(x^i_j)\}+\max\{0,g^i_{2j}(x^i_j)\})$,~$H(x)=\|h(x)\|$;
$\xi:[0,+\infty)\rightarrow[0,1]$ is given below
\begin{equation*}\label{gs3}
\xi(t)= \left\{\begin{aligned}
&0,~~~~~~~~~~~~~~~~~t>1,\\
&1-\nu,
~~~~~~~0\leq t\leq1.\\
\end{aligned}
\right.
\end{equation*}
\begin{rem}
    Although the non-monotone game \eqref{wenti1} has been transformed as an optimization problem \eqref{wenti3} with a continuous, differentiable cost function, the non-convexity of the cost function obstructs the global convergence of most approaches \cite{li2019distributed,jiang2022second,wei2023distributed}. Even for the convergence to critical point, the second-order sufficiency conditions required by Lagrangian methods \cite{xia2005recurrent,xia2023collaborative,liu2016collective} are not satisfied for such cost function. Thus, in this paper, requiring only the boundedness of $\Omega$, the ANA \eqref{gs1} with adaptive penalty method is proposed and the global convergence to the critical point set is demonstrated. Then, based on a hybrid swarm intelligence algorithm under the framework of CNA, we will prove its global optimal solution search ability with probability one.
\end{rem}
\subsection{Finite-Time Convergence to the Constraint Set}
Some important properties of ANA \eqref{gs1} are analyzed before the discussion on its convergence to the critical point set.

\begin{thm}\label{thm11}
For any $x(0)\in\mathcal{R}^{m}$, the state $x(t)$ of ANA \eqref{gs1} exists globally and is bounded.
\end{thm}
\begin{proof}
Note that $\xi$ is continuous on $\mathcal{R}$ and $\partial G(x)$ is upper semicontinuous with non-empty convex compact values. Thus, ANA \eqref{gs1} is a set-valued map with non-empty convex compact values. Based on Theorem 2.3 in \cite{aubin2009differential}, a local state $x(t)$ exists for $t\in[0,T_{\max})$. Then, for the local state $x(t)$, there is $\kappa(t)\in\partial G(x(t))$ and $\eta(t)\in\partial H(x(t))$ satisfying
\begin{align}\label{gs2122}
\dot{x}(t)= & -\xi(G(x(t)))\nabla_{x}\tilde{Q}(x)-\zeta(x(t))(\kappa(t)+\zeta(x(t))\eta(t)),
\end{align}
for a.e. $t\in[0,T_{\max})$. Based on the boundedness of $\Omega$, we can find an $R_{1}>0$ satisfying
\begin{equation}\label{gs6}
  \Omega\subset\{x:0\leq G(x)\leq1\}\subset \mathbb{B}(\bar{x},R_{1}),
\end{equation}
in which $\mathbb{B}(\bar{x},R_{1})=\{x:\|x-\bar{x}\|<R_{1}\}$ with $\bar{x}\in\text{int}(\Omega)\cap\Psi$.
According to \eqref{gs6}, there exists an $R_{2}$ satisfying that
\begin{equation}\label{gs7}
 \{\Omega\cup\{x(0)\}\}\subseteq \mathbb{B}(\bar{x},R_{2}).
\end{equation}
Then, considering the definition of $\xi(\cdot)$, we have
$\xi(G(\tilde{x}))=0,~\text{when}~G(\tilde{x})>1,$ i.e.,
\begin{equation}\label{eq5}
 \xi(G(\tilde{x}))=0,~~\text{if}~\tilde{x}\notin \mathbb{B}(\bar{x},R_{2}).
\end{equation}
Then, we demonstrate that for any $t\in[0,T_{\max})$, we can get $x(t)\in \mathbb{B}(\bar{x},R_{2})$. Using proof by contradiction, we assume we can find $T_{0},~\tau_{0}>0$ such that $\|x(T_{0})-\bar{x}\|=R_{2}$ and $\|x(t)-\bar{x}\|>R_{2}$ for any $t\in(T_{0},T_{0}+\tau_{0}]$. Next, for any $t\in(T_{0},T_{0}+\tau_{0}]$, there is
\begin{equation}\label{eq6}
\xi(G(x(t)))=0.
\end{equation}
Based on (\ref{eq6}), differentiating $\|x(t)-\bar{x}\|^{2}/2$, we can get
\begin{align}\label{gs9}
\frac{\mathrm{d}}{\mathrm{d}t}\frac{\|x(t)-\bar{x}\|^{2}}{2} =&\zeta(x(t))(\bar{x}-x(t))^{\top}(\kappa(t)+\zeta(x(t))\eta(t)),
\end{align}
for a.e. $t\in(T_{0},T_{0}+\tau_{0}]$.
Based on the convexity of $g^i_{1j}(x)$ and $g^i_{2j}(x)$, we can derive that $G(x)$ is convex. Thus, (\ref{gs9}) can be simplified into
\begin{equation*}
  \begin{split}
  &\frac{\mathrm{d}}{\mathrm{d}t}\frac{\|x(t)-\bar{x}\|^{2}}{2}\\
\leq&\zeta(x(t))(G(\bar{x})-G(x(t)))+(\zeta(x(t)))^{2}(H(\bar{x})-H(x(t)))\\
 =&-\zeta(x(t))(G(x(t)))-(\zeta(x(t)))^{2}H(x(t))\leq0.
\end{split}
\end{equation*}
This means that $\|x(T_{0}+\tau_{0})-\bar{x}\|\leq\|x(T_{0})-\bar{x}\|=R_{2},$
which is contrary to $\|x(T_{0}+\tau_{0})-\bar{x}\|>R_{2}$. Thus, for any $t\in[0,T_{\max})$, there is $x(t)\in \mathbb{B}(\bar{x},R_{2})$. Based on the Extension Theorem in \cite{aubin2009differential}, global existence and boundedness of state $x(t)$ are guaranteed.
\end{proof}
\begin{thm}\label{thm22}
For any $x(0)\in\mathcal{R}^m$, the state of ANA \eqref{gs1} will enter the constraint set in finite time and remain there thereafter.
\end{thm}
\begin{proof}
According to Theorem \ref{thm11}, there exists an $R_{2}>0,$ satisfying ${x(t)}\in \mathbb{B}(\bar{x},R_{2})$ for any $ t\geq0$. Also, the state $x(t)$ satisfies \eqref{gs2122} for a.e. $t\geq0$.
Since $x(t)$ is bounded and $\tilde{Q}$ is composed of polynomial functions, there exists an $l_{Q}>0$ such that
\begin{equation}\label{gs10}
  \|\nabla_{x}\tilde{Q}(x)\|\leq l_{Q}.
\end{equation}
As $x(t)$ is bounded and $\partial G$ is upper semicontinuous, we can find an $l_{G}>0$ such that
\begin{equation}\label{gs11}
  \|\partial G(x(t))\|:=\sup\{\|\kappa\|:\kappa\in\partial G(x)\}\leq l_{G},~\forall t\geq0.
\end{equation}
For convenience, denote $$C=\left[\begin{array}{llll}
\overbrace{1 \cdots 1}^{m_1} & & & \\
&\overbrace{1 \cdots 1}^{m_2} & & \\
& & \ddots & \\
& & & \overbrace{1 \cdots 1}^{m_N}
\end{array}\right],$$ and $d=\col(1,\cdots,1)$. Then, we can get $h(x)=Cx-d$ and matrix $C$ is full row rank. Denote $\lambda_{\min}$ as the minimal eigenvalue of $CC^{{\rm T}}.$ We can obtain that $\lambda_{\min}>0$. Selecting a sufficiently large constant $\bar{T}>0$, one can claim that a $T_{1}\geq\bar{T}$ exists such that $x(T_{1})\in\Omega\cap\Psi$. If it is incorrect, one obtains $x(t)\in \mathbb{B}(\bar{x},R_{2})\setminus(\Omega\cap\Psi)$ for any $t\geq\bar{T}$, which means $\zeta(x(t))\rightarrow+\infty$ when $t\rightarrow\infty$. Thus, we have
\begin{equation}\label{gs13}
  \left\{\begin{split}
     &b_1:=\zeta(x(\bar{T}))^{-1}l_{Q}-\frac{\bar{G}}{R_{2}}<0, \\
    &b_2:=l_{Q}+\zeta(x(\bar{T}))l_{G}-(\zeta(x(\bar{T})))^{2}\sqrt{\lambda_{\min}}<0,
   \end{split}
   \right.
\end{equation}
where $\bar{G}=\min_{i\in\mathbb{V},j\in\mathbb{M}_i}\{-g^{i}_{1j}(\bar{x}),-g^{i}_{2j}(\bar{x})\}>0$. Then, define
\begin{equation*}
\begin{split}
 V(t)=&\zeta(x(t))^{-1}G(x(t))+H(x(t)).
\end{split}
\end{equation*}
Then, we can prove that, for a.e. $t\geq\bar{T}$, there exists a $b_{0}>0$ such that
\begin{equation}\label{gs14}
  \frac{\mathrm{d}}{\mathrm{d}t}V(t)\leq-b_{0}.
\end{equation}
We derive for a.e. $t\geq\bar{T}$,
\begin{align}\label{gs15}
\nonumber\frac{\mathrm{d}}{\mathrm{d}t}V(t)=&-\zeta(x(t))^{-2}\sign{\varepsilon(x(t))}G(x(t))\\
&\nonumber+\zeta(x(t))^{-1}\langle\kappa(t),\dot{x}(t)\rangle+\langle\eta(t),\dot{x}(t)\rangle\\
\nonumber\leq&-\zeta(x(t))^{-1}\langle\zeta(x(t))\eta(t)+\kappa(t),\xi(G(x))\nabla_{x}\tilde{Q}(x)\\
&+\zeta(x(t))\kappa(t)+(\zeta(x(t)))^2\eta(t)\rangle.
\end{align}
Clearly, one can derive
\begin{equation}\label{2-1}
\mathbb{B}(\bar{x},R_{2})\setminus(\Omega\cap\Psi)=\{\mathbb{B}(\bar{x},R_{2})\setminus\Omega\}\cup\{\Omega\setminus\Psi\}.
\end{equation}
Therefore, one can demonstrate \eqref{gs14} merely by discussing the two cases below
\begin{equation*}
  \left\{\begin{split}
     \textit{\textbf{Case~i}}:~& x(t)\in \mathbb{B}(\bar{x},R_{2})\setminus\Omega, ~t\geq\bar{T},\\
     \textit{\textbf{Case~ii}}:~& x(t)\in\Omega\setminus\Psi, ~~~~~~~~~~t\geq\bar{T}.
   \end{split}
   \right.
\end{equation*}
\subsubsection*{$\textbf{Case~i}$} For $x(t)\in \mathbb{B}(\bar{x},R_{2})\setminus\Omega$, as $H(\cdot)$ is convex, we have
\begin{equation*}
\begin{split}
  &\langle\kappa(t)+\zeta(x(t))\eta(t),x(t)-\bar{x}\rangle\\
  \geq&\langle\kappa(t),x(t)-\bar{x}\rangle+\zeta(x(t))\big(H(x(t))-H(\bar{x})\big).
\end{split}
\end{equation*}
Based on the coerciveness of $G$, we can get
\begin{equation}\label{gs17}
\begin{split}
 \langle\kappa(t)+\zeta(x(t))\eta(t),x(t)-\bar{x}\rangle\geq\langle\kappa(t),x(t)-\bar{x}\rangle>\bar{G}.
\end{split}
\end{equation}
Thus, it yields
\begin{equation}\label{gs18}
  \big\|\kappa(t)+\zeta(x(t))\eta(t)\big\|\geq\frac{\bar{G}}{R_{2}}.
\end{equation}
Thus, from \eqref{gs15}, for a.e. $t\geq\bar{T}$, one derives that
\begin{equation*}
  \begin{split}
    &~\frac{\mathrm{d}}{\mathrm{d}t}V(t)\\
     \leq&-\zeta(x(t))^{-1}\langle\kappa(t)+\zeta(x(t))\eta(t),\xi\big(G\big(x(t)\big)\big)\nabla_{x}\tilde{Q}(x)\rangle\\
    &-\|\kappa(t)+\zeta(x(t))\eta(t)\|^{2} \\
       \leq&\|\kappa(t)+\zeta(x(t))\eta(t)\|\bigg(\zeta(x(t))^{-1}\|\nabla_{x}\tilde{Q}(x)\|\\
       &-\|\kappa(t)+\zeta(x(t))\eta(t)\|\bigg)\\
       \leq&\frac{\bar{G}}{R_{2}}\left(\zeta(x(\bar{T}))^{-1}l_{Q}-\frac{\bar{G}}{R_{2}}\right)<0.
   \end{split}
\end{equation*}
\subsubsection*{$\textbf{Case~ii}$}
For $x(t)\in\Omega\setminus\Psi$, we have $Cx(t)\neq d$.
As $x(t)\in\Omega$, there is $\mathbf{0}_m\in\partial G(x(t))$. Thus, for a.e. $t\geq\bar{T}$, \eqref{gs15} is reformulated as
\begin{equation*}
  \begin{split}
     &\frac{\mathrm{d}}{\mathrm{d}t}V(t)\\
     =&-\langle\eta(t),\xi\big(G\big(x(t)\big)\big)\nabla_{x}\tilde{Q}(x)+\zeta(x(t))\\
     &\times(\kappa(t)+\zeta(x(t))\eta(t))\rangle\\
       \leq&\|\eta(t)\|\left(\|\nabla_{x}\tilde{Q}(x)\|+\zeta(x(t))\|\kappa(t)\|-(\zeta(x(t)))^2\|\eta(t)\|\right).
   \end{split}
\end{equation*}
Based on \eqref{gs10} and \eqref{gs11}, we can get
\begin{equation}\label{gs20}
  \frac{\mathrm{d}}{\mathrm{d}t}V(t)\leq\|\eta(t)\|\bigg(l_{Q}+\zeta(x(t)) l_{G}-(\zeta(x(t)))^{2}\|\eta(t)\|\bigg).
\end{equation}
Since $\lambda_{\min}$ is the minimal eigenvalue, one derives
\begin{equation}\label{gs21}
  \|\eta(t)\|^{2}=\frac{(Cx-d)^{\top}CC^{\top}(Cx-d)}{\|Cx-d\|^{2}}\geq\lambda_{\min}.
\end{equation}
It means that $\|\eta(t)\|\geq\sqrt{\lambda_{\min}}.$ Based on \eqref{gs13} and \eqref{gs21}, for a.e. $t\geq\bar{T}$, formula \eqref{gs20} can be scaled as
\begin{equation}\label{gs22}
\begin{split}
  \frac{\mathrm{d}}{\mathrm{d}t}V(t)\leq&\sqrt{\lambda_{\min}}\bigg(l_{Q}+\zeta(x(\bar{T}))l_{G}\\
    &-(\zeta(x(\bar{T})))^{2}\sqrt{\lambda_{\min}}\bigg)<0.
\end{split}
\end{equation}
Let $$b_{0}=\min\{\bar{b}_{1},\bar{b}_{2}\},$$ in which $\bar{b}_{1}=b_{1}\bar{G}/R_{2}$ and $\bar{b}_{2}=\sqrt{\lambda_{\min}}b_{2}$. Therefore, for a.e. $t\geq\bar{T}$, it is clear that \eqref{gs14} holds. Next, integrating \eqref{gs14} from $\bar{T}$ to $t$, one derives
\begin{equation*}
  V(t)\leq V(\bar{T})-b_{0}(t-\bar{T}),
\end{equation*}
which is contrary to the nonnegativity of $V(t)$ for a sufficiently large $t$. Therefore, we can find an $T^{*}\geq\bar{T}$ satisfying $x(T^{*})\in\Omega\cap\Psi$.

At last, we demonstrate for any $t\geq T^{*}$, there is $x(t)\in\Omega\cap\Psi$. If it is incorrect, we can find $t_{1},t_{2}>0$ such that $t_{2}\geq t_{1}\geq T^{*}$, $x(t_{1})\in\Omega\cap\Psi$ and $x(t)\notin\Omega\cap\Psi$ for any $t\in(t_{1},t_{2}]$. According to a similar process, \eqref{gs14} still holds for $t\in(t_{1},t_{2}]$.
Integrating \eqref{gs14} from $t_{1}$ to $t$, one derives
\begin{equation*}
  V(t)\leq V(t_{1})-b_{0}(t-t_{1})=-b_{0}(t-t_{1})<0,
\end{equation*}
which is contrary to $V(t)\geq0$ for any $t\geq0$. Therefore, the state $x(t)$ will enter the constraint set in finite time and stay there thereafter.
\end{proof}
\begin{rem} 
Based on the adaptive penalty method, the state of neurodynamic approach \eqref{gs1} will reach the constraint set in finite time, where almost all of the researches by lagrangian methods can not ensure. Also, different from the exact penalty method in \cite{yan2016collective}, the penalty parameters in ANA \eqref{gs1} can be autonomously computed based on the intrinsic characteristic of the constraints.
\end{rem}
\subsection{Convergence Analysis}
Before analyzing the convergence of ANA \eqref{gs1}, the following lemma is necessary.
\begin{lem}\label{g90}
The derivative $\dot{x}(t)$ of the state $x(t)$ of ANA \eqref{gs1} is a.e. convergent to $0$, i.e., $\pi$-$\lim_{t\rightarrow\infty}\dot{x}(t)=0$ with $\pi$ being the Lebesgue measure.
\end{lem}
\begin{proof}
    From Theorem \ref{thm22}, one derives for any $t\geq T^*$, we have $x(t)\in\Omega\cap\Psi$. It implies \begin{equation*}
  G\big(x(t)\big)\equiv0,~H\big(x(t)\big)\equiv0,~\forall t\geq T^*.
\end{equation*}
Therefore, for a.e. $t\geq T^*$, we obtain
\begin{equation*}
\xi\big(G(x(t))\big)=1,~\frac{\mathrm{d}}{\mathrm{d}t}G\big(x(t)\big)=0,~\frac{\mathrm{d}}{\mathrm{d}t}H(x(t))=0.
\end{equation*}
From \eqref{gs2122}, one derives
\begin{equation*}
  \begin{split}
  &\frac{\mathrm{d}}{\mathrm{d}t}\tilde{Q}(x(t))\\
  =&\frac{\mathrm{d}}{\mathrm{d}t}\tilde{Q}(x(t))+\zeta(x(t))\frac{\mathrm{d}}{\mathrm{d}t}G\big(x(t)\big)+(\zeta(x(t)))^{2}\frac{\mathrm{d}}{\mathrm{d}t}H\big(x(t)\big)\\
       =&-\|\dot{x}(t)\|^{2}.
   \end{split}
\end{equation*}
Integrating it from $T^*$ to $\infty$ gives
\begin{align}\label{shiziq} \lim_{t\rightarrow\infty}\big(\tilde{Q}(x(t))\big)-\tilde{Q}(x(T^*))=-\int_{T^*}^{\infty}\|\dot{x}(t)\|^{2}\mathrm{d}t.
\end{align}
Based on \eqref{shiziq}, due to the boundedness of $x(t)$, we can find an $M\in[0,+\infty)$ satisfying
\begin{equation}\label{shiz2}
 \lim_{T\rightarrow\infty}\int_{T^*}^{T}\|\dot{x}(t)\|^{2}\mathrm{d}t=\int_{T^*}^{\infty}\|\dot{x}(t)\|^{2}\mathrm{d}t=M<\infty.
\end{equation}
If $\dot{x}(t)\nrightarrow0$ a.e., then for any $T>T^*$, we can find a $\bar{\phi}>0$ such that
$\pi\{t:\dot{x}(t)>\bar{\phi},~t\in[T,\infty)\}\geq\bar{\phi}.$ Next, we obtain
\begin{equation}\label{weizq}
\int_{T}^{\infty}\|\dot{x}(t)\|^{2}\mathrm{d}t>\bar{\phi}^{3}.
\end{equation}
However, from \eqref{shiz2}, one derives for any $\phi>0$, there is $T^{c}>T^*$ such that
\begin{equation}\label{shize}
  M-\int_{T^*}^{T^c}\|\dot{x}(t)\|^{2}\mathrm{d}t=\int_{T^c}^{\infty}\|\dot{x}(t)\|^{2}\mathrm{d}t\leq\phi.
\end{equation}
Letting $\phi=\bar{\phi}^{3}$, it gives that $\int_{T^c}^{\infty}\|\dot{x}(t)\|^{2}\mathrm{d}t\leq\bar{\phi}^{3}$. It is contrary to \eqref{weizq}. Thus, $\pi$-$\lim_{t\rightarrow\infty}\dot{x}(t)=0$.
\end{proof}
\begin{thm}\label{thm3}
The state $x(t)$ of ANA \eqref{gs1} is convergent to the critical point set of problem \eqref{wenti3}, that is, $\lim_{t\rightarrow\infty}\mathrm{dist}(x(t),\mathcal{C})=0.$
\end{thm}
\begin{proof}
Based on Theorem \ref{thm22}, we can find a $T^{*}>0$ such that for all $t\geq T^{*}$, we have $x(t)\in\Omega\cap\Psi$. It implies $\zeta(x(t))\equiv \theta$ with a fixed $\theta>0$, $G(x(t))=0$, and $\xi(G(x(t)))=1$, for all $t\geq T^{*}$. Thus, for a.e. $t\geq T^*$, the ANA \eqref{gs1} is rewritten by
 \begin{equation}\label{gs341}
  \dot{x}(t)=-\nabla_{x}\tilde{Q}(x(t))-\theta(\kappa(t)+\theta\eta(t)).
\end{equation}
As $x(t)$ is bounded, at least one Accumulation Point (AP) of $x(t)$ exists when $t\rightarrow\infty$. Denote $\mathcal{K}$ as the AP set. Thus, the only work left is to demonstrate $\mathcal{K}\subseteq\mathcal{C}$, which implies $\lim_{t\rightarrow\infty}\mathrm{dist}(x(t),\mathcal{C})=0$.

Based on the absolute continuity of x(t), according to \cite{aubin2009differential}, one can get for all $x^{*}\in\mathcal{K}$, $x^{*}$ is an almost AP of $x(t)$ when $t\rightarrow\infty$. Thus, we derive for any $\beta\in\{1,2,\cdots\}$, we can find a $T^{1}_{\beta}>0$ such that
\begin{equation}\label{shizd}
  \pi\big(\mathcal{H}_{\beta}\cap[T^{1}_{\beta},\infty)\big)=\infty,~\mathcal{H}_{\beta}=\{t:||x(t)-x^{*}||\leq1/\beta\}.
\end{equation}
Moreover, based on Lemma \ref{g90}, $\pi$-$\lim_{t\rightarrow\infty}\dot{x}(t)=0$. Thus, one derives for any $\beta\in\{1,2,\cdots\}$, there is a $T^{2}_{\beta}>0$ such that
\begin{equation}\label{shizg3}
  \pi\big(\mathcal{D}_{\beta}\cap[T^{2}_{\beta},\infty)\big)<1/\beta,~\mathcal{D}_{\beta}=\{t:||x(t)||>1/\beta\}.
\end{equation}
Combining \eqref{shizd} and \eqref{shizg3}, letting $T_{\beta}=\max\{T^{1}_{\beta},T^{2}_{\beta}\}$, one can obtain
$\pi(\mathcal{P}_{\beta})=\infty,$ in which $\mathcal{P}_{\beta}=(\mathcal{H}_{\beta}\backslash\mathcal{D}_{\beta})\cap[T_{\beta},\infty)$. In conclusion, there is a sequence $\{t_{\beta}\}\subseteq\bigcup_{\beta=1}^{\infty}\mathcal{P}_{\beta}$ and $\lim_{\beta\rightarrow\infty}t_{\beta}=\infty$ such that
\begin{equation}\label{shizi5r}
 \lim_{\beta\rightarrow\infty}\dot{x}(t_{\beta})=0,~\lim_{\beta\rightarrow\infty}x(t_{\beta})=x^{*}\in\Omega\cap\Psi.
\end{equation}
Then, based on \eqref{gs341}, it gives that
\begin{equation}\label{gs28}
  \dot{x}(t_{\beta})=-\nabla_{x}\tilde{Q}(x(t_\beta))-\theta(\kappa(t_{\beta})+\theta\eta(t_{\beta})).
\end{equation}
Next, according to \eqref{shizi5r}, we have
\begin{equation}\label{gs30}
\lim_{\beta\rightarrow\infty}\nabla_{x}\tilde{Q}(x(t_\beta))=\nabla_{x} \tilde{Q}(x^{*}).
\end{equation}
Combining the boundedness of $x(t_{\beta})$ and the upper semicontinuity of $\partial G(x)$, one derives that $\{\kappa(t_{\beta})\}$ is bounded. Thus, there is a subsequence $\{\kappa(t_{\beta_{k}})\}$ such that
\begin{equation*}
\lim\limits_{k\rightarrow+\infty}\kappa(t_{\beta_{k}})=\kappa^{*}\in\partial G(x^{*}).
\end{equation*}
Similarly, we obtain
\begin{equation*}
\lim\limits_{k\rightarrow+\infty}\eta(t_{\beta_{k}})=\eta^{*}\in\partial H(x^{*}).
\end{equation*}
Based on \eqref{gs28} and \eqref{gs30}, when $\beta\rightarrow\infty$, one derives
\begin{equation}\label{gs29}
\begin{split}
  \mathbf{0}_m\in&-\nabla_{x} \tilde{Q}(x^{*})-\theta\Big(\partial G(x^{*})+\theta\partial H(x^{*})\Big).
\end{split}
\end{equation}
It implies $x^{*}\in\mathcal{C}$. Thus, the state $x(t)$ of ANA \eqref{gs1} is convergent to the critical point set.
\end{proof}
\section{An Adaptive Collaborative Neurodynamic Approach for Global Optimal Solutions}
To ensure global and exact NE computation, an ACNA is developed by combining ANA \eqref{gs1} with particle swarm algorithm.

In the ACNA, a group of ANA \eqref{gs1} will be employed to search for the critical points, and a meta-heuristic principle will be utilized to re-initialize the states based on historical experiences and the obtained critical points. 

In this paper, the particle swarm algorithm is employed as the meta-heuristic principle, which is presented as below: for $i\in\{1,\cdots,r\}$,
\begin{equation}\label{pso1}
\begin{aligned}
\mathbf{v}_i(k+1)=&\alpha \mathbf{v}_i(k)+c_1 l_1\left(p_{i,\text {pbest }}(k)-\mathbf{x}_i(k)\right) \\
&+c_2 l_2\left(p_{\text {gbest }}(k)-\mathbf{x}_i(k)\right), \\
\mathbf{x}_i(k+1)=&\mathbf{x}_i(k)+\mathbf{v}_i(k+1),
\end{aligned}
\end{equation}
where $r$ denotes the number of ANA \eqref{gs1} employed in the system, $k$ denotes the iteration step, $\mathbf{x}_i$ is the re-initialized state of $i$th ANA \eqref{gs1}, $\mathbf{v}_i$ denotes velocity of the $i$th ANA \eqref{gs1}, $p_{i,\text {pbest }}$ represents the best state of $i$th ANA \eqref{gs1}, $p_{\text {gbest }}$ represents system's best state, $\alpha,~c_1$, and $c_2$ are the weights, $l_1$ and $l_2$ are random constants remained in $[0,1]$. The iteration of the ACNA is shown as follows:
\begin{enumerate}[Step 1:]
\item Input $\alpha,~ c_1,~ c_2$, error tolerance $\epsilon$, termination criterion $\tilde{T}$, and the number of ANA \eqref{gs1} $r$.
\item Initialize $k=0, T=0, \mathbf{x}_i(0)$, and $\mathbf{v}_i(0)$.\\
Initialize Individually best state $p_{i,\text {pbest }}(0)=\mathbf{x}_i(0)$.\\
Initialize group best state $p_{\text {gbest }}(0)=\mathbf{x}_{i^*}(0)$\\
with $i^*=\arg \min _i \tilde{Q}\left(\mathbf{x}_i(0)\right)$.
\item If $T \leq \tilde{T}$, Run main loop (For $i\in\{1,\cdots,r\}$):
\begin{enumerate}[Step 3-1:]
\item Run the $i$th ANA \eqref{gs1} from $\mathbf{x}_i(k)$ until reaching the critical point $\bar{\mathbf{x}}_i(k)$.
\item If $\tilde{Q}\left(p_{i,\text {pbest }}(k)\right)>\tilde{Q}\left(\bar{\mathbf{x}}_i(k)\right)$, then $p_{i,\text {pbest }}(k+1)=\bar{\mathbf{x}}_i(k)$; otherwise, $p_{i,\text {pbest }}(k+1)=p_{i,\text {pbest }}(k)$.
\item If $\tilde{Q}\left(p_{\text {gbest }}(k)\right)>\tilde{Q}\left(p_{i^*,\text {pbest }}(k+1)\right)$ with $i^*=\arg \min _i \tilde{Q}\left(p_{i,\text {pbest }}(k+1)\right)$, then $p_{\text {gbest }}(k+1)=p_{i^*,\text {pbest }}(k+1)$; otherwise, $p_{\text {gbest }}(k+1)=p_{\text {gbest }}(k)$.
\item If $\left\|\tilde{Q}\left(p_{\text {gbest }}(k+1)\right)-\tilde{Q}\left(p_{\text {gbest }}(k)\right)\right\| \leq \epsilon$, then $T=T+1$; otherwise, $T=0$.
\item Obtain $\mathbf{x}_i(k+1)$ by \eqref{pso1}, $k=k+1$, then move to Step 3-1.
\end{enumerate}
\item If $T>\tilde{T}$, then Stop.
\end{enumerate}

The ACNA is a stochastic approach, which inherits the a.e. convergence from its precursors to ensure global optimal solution search \cite{yan2016collective}.
\begin{thm}
Based on the Steps 1-4, the ACNA is convergent to a global optimal solution of optimization problem \eqref{wenti3} with probability one.
\end{thm}
\begin{proof}
    It can be proved by a similar process as the proof in \cite{yan2016collective}. Based on Theorem \ref{thm3} and Step 3, the cost function value is monotonically nonincreasing with respect to the group best state sequence $\{p_{\text {gbest }}(k)\}_{k=1}^{\text{end}}$. Let $\mathcal{C}_i(k)$ be the set of critical point by running the $i$th ANA \eqref{gs1}. According to Theorem \ref{thm3}, one can get the global optimal solution $x_{\text {global }}^* \in \mathcal{C}_i(k)$ for $i\in\{1,\cdots,r\}$. Besides, Lebesgue measure of each attracted region of each $x_{\text {critical }}^* \in \mathcal{C}_i(k)$ or $x_{\text {global }}^* \in \mathcal{C}_i(k)$ is greater than 0. According to \cite{yan2016collective}, the ACNA is convergent to a global minima of problem \eqref{wenti3} with probability one.
\end{proof}
\section{Numerical Examples}
Consider a standard rock-paper-scissor game with three players. The payoff matrix is formulated as follows:
\begin{table}[htpb]
\centering
\caption{Payoff matrix for rock-paper-scissor game}
\begin{threeparttable}
      \begin{tabular}{|c|c|c|c|}
\hline \diagbox{P1}{P2}& R & P & S \\
\hline R & (0,0,0) & (-1,2,-1) & (1,-2,1) \\
\hline P & (2,-1,-1) & (1,1,-2) & (0,0,0) \\
\hline S & (-2,1,1) & (0,0,0) & (-1,-1,2) \\
\hline
\end{tabular}
      \begin{tablenotes}
		\item P3: R
     \end{tablenotes}
\end{threeparttable}\\
\begin{threeparttable}
      \begin{tabular}{|c|c|c|c|}
\hline \diagbox{P1}{P2}& R & P & S \\
\hline R & (-1,-1,2) & (-2,1,1) & (0,0,0) \\
\hline P & (1,-2,1) & (0,0,0) & (-1,2,-1) \\
\hline S & (0,0,0) & (2,-1,-1) & (1,1,-2) \\
\hline
\end{tabular}
      \begin{tablenotes}
		\item P3: P
     \end{tablenotes}
\end{threeparttable}\\
\begin{threeparttable}
      \begin{tabular}{|c|c|c|c|}
\hline \diagbox{P1}{P2}& R & P & S \\
\hline R & (1,1,-2) & (0,0,0) & (2,-1,-1) \\
\hline P & (0,0,0) & (-1,-1,2) & (-2,1,1) \\
\hline S & (-1,2,-1) & (1,-2,1) & (0,0,0) \\
\hline
\end{tabular}
      \begin{tablenotes}
		\item P3: S
     \end{tablenotes}
\end{threeparttable}
\end{table}

In the above normal-form game, the only mixed-strategy NE is $((\frac{1}{3},\frac{1}{3},\frac{1}{3}),(\frac{1}{3},\frac{1}{3},\frac{1}{3}),(\frac{1}{3},\frac{1}{3},\frac{1}{3}))$. Then, the payoff functions of the three players can be calculated as:\\
$u^1(x)=-x^1_Rx^2_Rx^3_P+x^1_Rx^2_Rx^3_S-x^1_Rx^2_Px^3_R-2x^1_Rx^2_Px^3_P+x^1_Rx^2_Sx^3_R+2x^1_Rx^2_Sx^3_S+2x^1_Px^2_Rx^3_R+x^1_Px^2_Rx^3_P+x^1_Px^2_Px^3_R-x^1_Px^2_Px^3_S-x^1_Px^2_Sx^3_P-2x^1_Px^2_Sx^3_S-2x^1_Sx^2_Rx^3_R-x^1_Sx^2_Rx^3_S+2x^1_Sx^2_Px^3_P+x^1_Sx^2_Px^3_S-x^1_Sx^2_Sx^3_R+x^1_Sx^2_Sx^3_P,$\\
$u^2(x)=-x^2_Rx^1_Rx^3_P+x^2_Rx^1_Rx^3_S-x^2_Rx^1_Px^3_R-2x^2_Rx^1_Px^3_P+x^2_Rx^1_Sx^3_R+2x^2_Rx^1_Sx^3_S+2x^2_Px^1_Rx^3_R+x^2_Px^1_Rx^3_P+x^2_Px^1_Px^3_R-x^2_Px^1_Px^3_S-x^2_Px^1_Sx^3_P-2x^2_Px^1_Sx^3_S-2x^2_Sx^1_Rx^3_R-x^2_Sx^1_Rx^3_S+2x^2_Sx^1_Px^3_P+x^2_Sx^1_Px^3_S-x^2_Sx^1_Sx^3_R+x^2_Sx^1_Sx^3_P,$\\
$u^3(x)=-x^3_Rx^1_Rx^2_P+x^3_Rx^1_Rx^2_S-x^3_Rx^1_Px^2_R-2x^3_Rx^1_Px^2_P+x^3_Rx^1_Sx^2_R+2x^3_Rx^1_Sx^2_S+2x^3_Px^1_Rx^2_R+x^3_Px^1_Rx^2_P+x^3_Px^1_Px^2_R-x^3_Px^1_Px^2_S-x^3_Px^1_Sx^2_P-2x^3_Px^1_Sx^2_S-2x^3_Sx^1_Rx^2_R-x^3_Sx^1_Rx^2_S+2x^3_Sx^1_Px^2_P+x^3_Sx^1_Px^2_S-x^3_Sx^1_Sx^2_R+x^3_Sx^1_Sx^2_P.$

It is clear that the payoff functions are non-convex and the corresponding pseudo-gradient is non-monotone. Thus, most of the gradient descent-based approaches lose their efficiency \cite{he2022distributed,yuan2024event,Rao2023,Ye2020,li2024logical,huang2024distributed,deng2022nash,Wu2020,Zhou2005,Bianchi2021,deng2018distributed,meng2022attack,deng2021distributed,franci2021training,Wang2021}. To testify the effectiveness of the ACNA, the following parameter setting is given: $$\alpha=0.4+((0.9 - 0.4)*(1.0-(\frac{\text{iteration}}{\text{max\_iteration}}))),$$
$c_1=c_2=2,~\epsilon=0.1,~\tilde{T}=100,~r=10.$ Then, randomly initialize player's state $x^i_j\in[-10,10]$ for $i\in\{1,2,3\}$, $j\in\{P,R,S\}$. Run the ACNA and the results are given in Figs. \ref{tu2a}-\ref{tu2d}.
\begin{figure}[bht]
\begin{center}
	\subfigure[The state of ACNA.\label{tu2a}]
	{\includegraphics[width=0.492\linewidth]{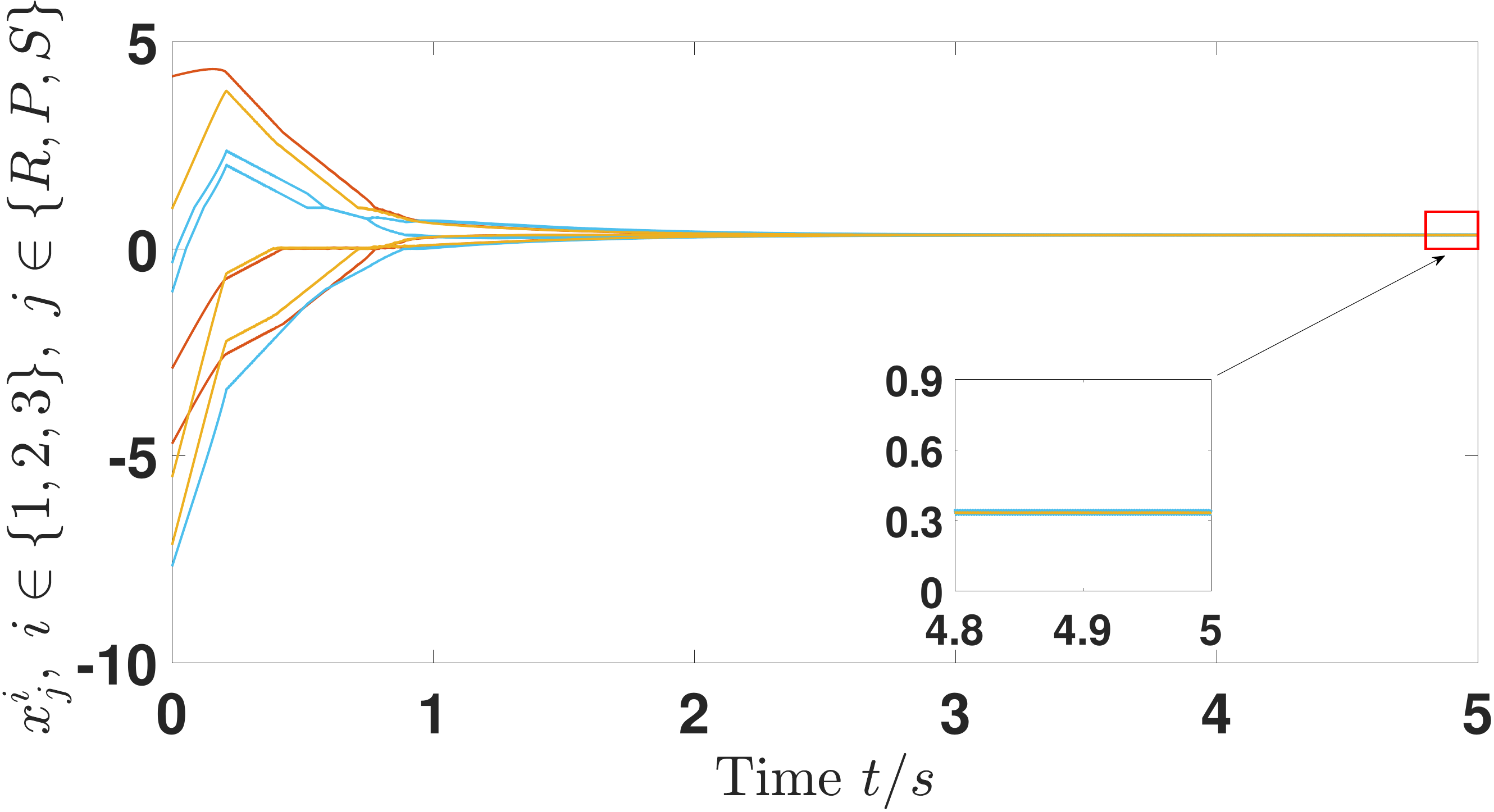}}
	\subfigure[The cost value of the transformed optimization problem $\tilde{Q}(x)$.\label{tu2b}]
	{\includegraphics[width=0.492\linewidth]{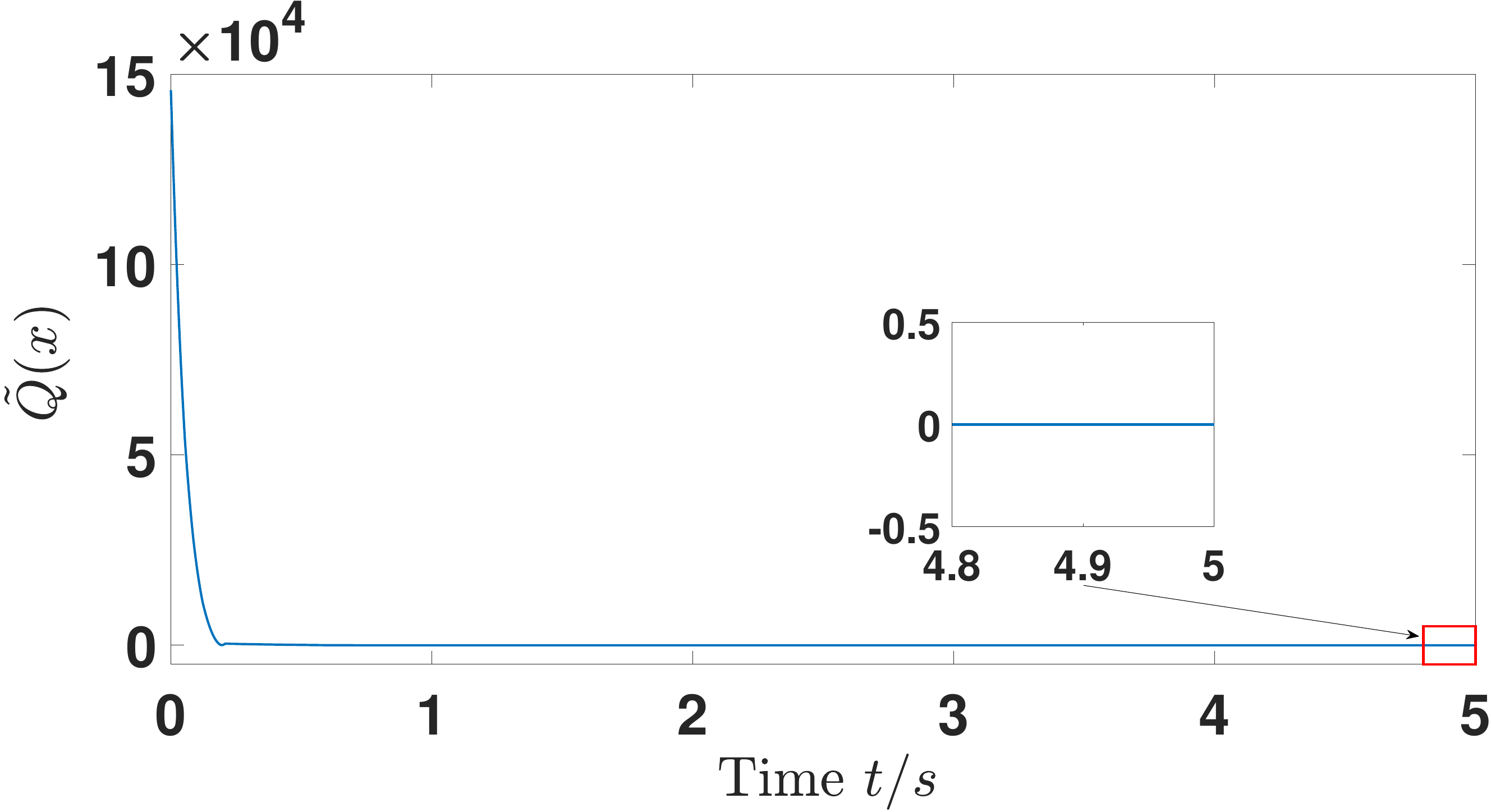}}

	\subfigure[The value of $G(x)$.\label{tu2c}]
	{\includegraphics[width=0.492\linewidth]{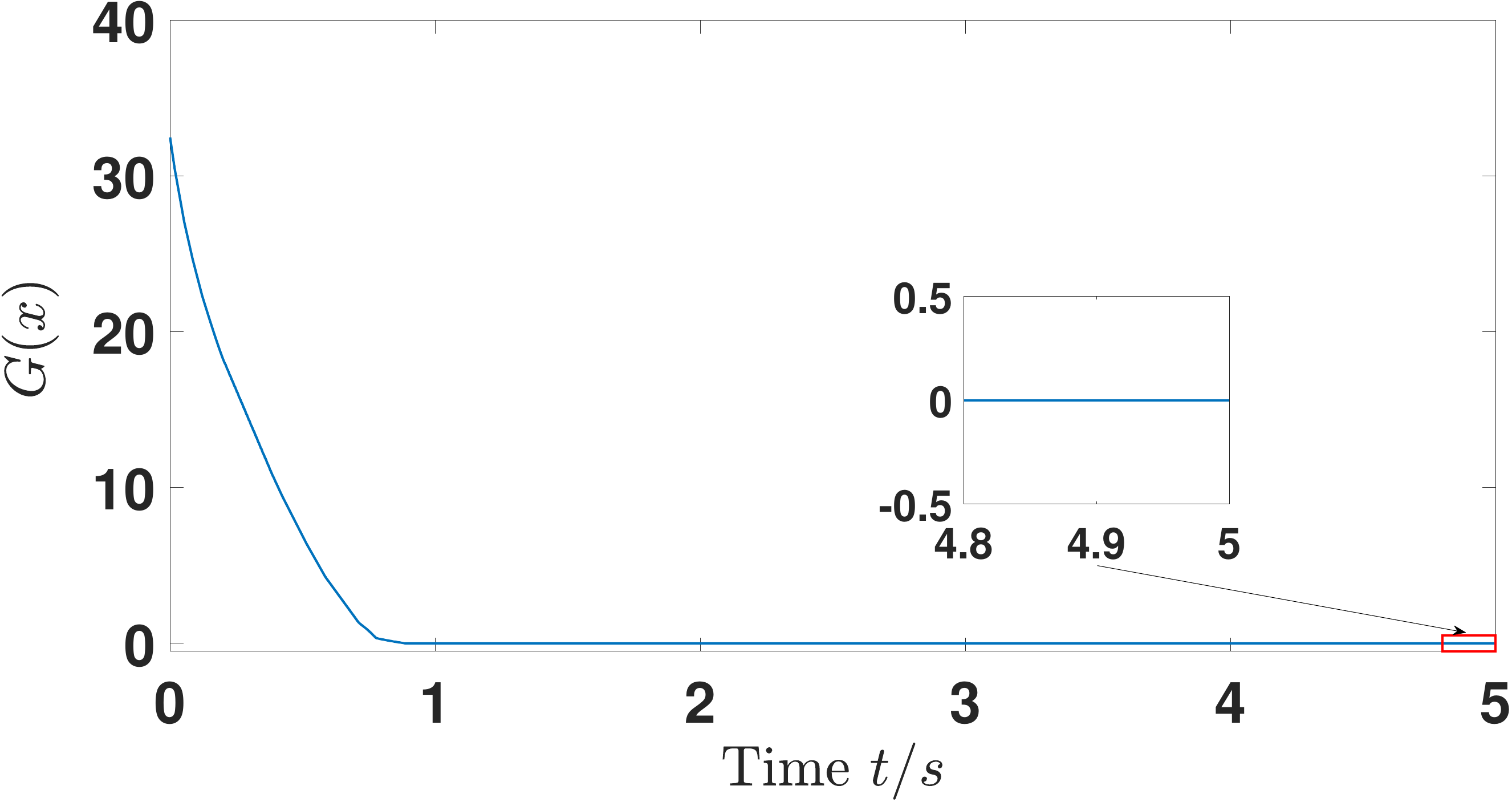}}
	\subfigure[The value of $H(x)$.\label{tu2d}]
	{\includegraphics[width=0.492\linewidth]{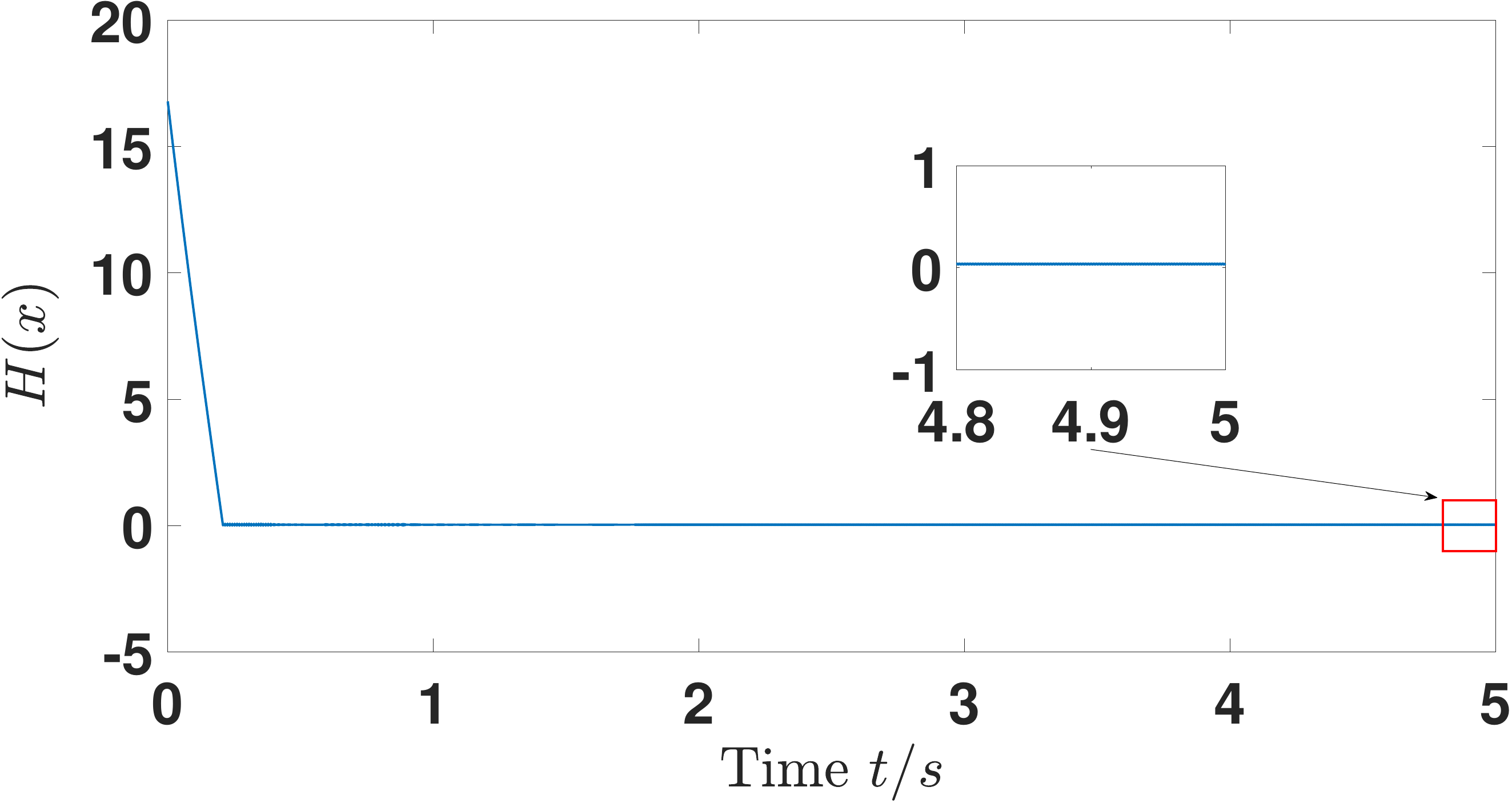}}
\end{center}
\end{figure}

Fig. \ref{tu2a} displays the evolution of players' strategies by the ACNA, which clearly shows that all the players' strategies converge to $1/3$. Also, the cost value of the transformed optimization problem $\tilde{Q}(x)$ converges to zero as presented in Fig. \ref{tu2b}. Then, to testify that players' strategies enter the constraint set, Figs. \ref{tu2c} and \ref{tu2d} are depicted to show the values of $G(x)$ and $H(x)$, which are both convergent to zero. It demonstrates that the ACNA is able to ensure global and exact NE computation for $N$-player non-monotone normal-form game.

\section{Conclusion}
In this paper, a novel ACNA is introduced for efficiently computing the NE in general $N$-player normal-form games with mixed strategies. By transforming the normal-form game with non-monotone payoff functions into an optimization problem and leveraging the adaptive penalty method, ACNA ensures the global and exact computation of an NE without relying on the second-order sufficiency conditions of Lagrangian methods or the complex penalty parameter estimation required by exact penalty methods. The use of particle swarm algorithm further guarantees global convergence to an exact NE with probability one. Finally, numerical simulation is provided to demonstrate the effectiveness of the proposed method, confirming its potential as a robust and efficient tool for solving large-scale games. In the future, the authors will attempt to apply the ACNA to search for more kinds of equilibrium, such as perfect equilibrium or proper equilibrium.
\bibliographystyle{IEEEtran}    
\bibliography{cdc}

\begin{thebibliography}{10}
\providecommand{\url}[1]{#1}
\csname url@samestyle\endcsname
\providecommand{\newblock}{\relax}
\providecommand{\bibinfo}[2]{#2}
\providecommand{\BIBentrySTDinterwordspacing}{\spaceskip=0pt\relax}
\providecommand{\BIBentryALTinterwordstretchfactor}{4}
\providecommand{\BIBentryALTinterwordspacing}{\spaceskip=\fontdimen2\font plus
\BIBentryALTinterwordstretchfactor\fontdimen3\font minus
  \fontdimen4\font\relax}
\providecommand{\BIBforeignlanguage}[2]{{%
\expandafter\ifx\csname l@#1\endcsname\relax
\typeout{** WARNING: IEEEtran.bst: No hyphenation pattern has been}%
\typeout{** loaded for the language `#1'. Using the pattern for}%
\typeout{** the default language instead.}%
\else
\language=\csname l@#1\endcsname
\fi
#2}}
\providecommand{\BIBdecl}{\relax}
\BIBdecl

\bibitem{bierman1998game}
H.~S. Bierman and L.~F. Fernandez, \emph{Game Theory with Economic
  Applications}.\hskip 1em plus 0.5em minus 0.4em\relax Addison-Wesley Reading,
  MA, 1998.

\bibitem{amit2024games}
R.~Amit, ``Games in normal form: Applications in {OM},'' in \emph{Game Theory
  with Applications in Operations Management}.\hskip 1em plus 0.5em minus
  0.4em\relax Springer, 2024, pp. 57--84.

\bibitem{ritzberger1995evolutionary}
K.~Ritzberger and J.~W. Weibull, ``Evolutionary selection in normal-form
  games,'' \emph{Econometrica: Journal of the Econometric Society}, pp.
  1371--1399, 1995.

\bibitem{li2024survey}
H.~Li, W.~Huang, Z.~Duan, D.~H. Mguni, K.~Shao, J.~Wang, and X.~Deng, ``A
  survey on algorithms for {N}ash equilibria in finite normal-form games,''
  \emph{Computer Science Review}, vol.~51, p. 100613, 2024.

\bibitem{nash1950equilibrium}
J.~F. Nash~Jr, ``Equilibrium points in n-person games,'' \emph{Proceedings of
  the National Academy of Sciences}, vol.~36, no.~1, pp. 48--49, 1950.

\bibitem{lemke1964equilibrium}
C.~E. Lemke and J.~T. Howson, Jr, ``Equilibrium points of bimatrix games,''
  \emph{Journal of the Society for industrial and Applied Mathematics},
  vol.~12, no.~2, pp. 413--423, 1964.

\bibitem{rosenmuller1971generalization}
J.~Rosenm{\"u}ller, ``On a generalization of the {L}emke--{H}owson algorithm to
  noncooperative {N}-person games,'' \emph{SIAM Journal on Applied
  Mathematics}, vol.~21, no.~1, pp. 73--79, 1971.

\bibitem{garcia1973simplicial}
C.~B. Garcia, C.~E. Lemke, and H.~L{\"u}ethi, ``Simplicial approximation of an
  equilibrium point for non-cooperative {N}-person games,'' \emph{Mathematical
  Programming}, pp. 227--260, 1973.

\bibitem{herings2001differentiable}
P.~J.-J. Herings and R.~J. Peeters, ``A differentiable homotopy to compute
  {N}ash equilibria of n-person games,'' \emph{Economic Theory}, vol.~18, pp.
  159--185, 2001.

\bibitem{cao2025efficient}
Y.~Cao, C.~Dang, and Y.~Yu, ``Efficient differentiable path-following methods
  for computing {N}ash equilibria,'' \emph{Computational Economics}, pp. 1--34,
  2025.

\bibitem{he2022distributed}
X.~He and J.~Huang, ``Distributed {N}ash equilibrium seeking and disturbance
  rejection for high-order integrators over jointly strongly connected
  switching networks,'' \emph{IEEE Transactions on Cybernetics}, 2022.

\bibitem{yuan2024event}
H.~Yuan, L.~Zhao, Y.~Yuan, and Y.~Xia, ``Event-triggered mechanism-based
  discrete-time {N}ash equilibrium seeking for graphic game with
  outlier-resistant {ESO},'' \emph{IEEE Transactions on Cybernetics}, pp.
  1--13, 2024.

\bibitem{Rao2023}
X.~Rao, W.~Xu, S.~Yang, and W.~Yu, ``A distributed coding-decoding-based {N}ash
  equilibrium seeking algorithm over directed communication network,''
  \emph{Science China Technological Sciences}, pp. 1--12, 2023.

\bibitem{Ye2020}
M.~Ye, ``Distributed robust seeking of {N}ash equilibrium for networked games:
  An extended state observer-based approach,'' \emph{IEEE Transactions on
  Cybernetics}, vol.~52, no.~3, pp. 1527--1538, 2020.

\bibitem{li2024logical}
C.~Li, A.~Li, Y.~Wu, and L.~Wang, ``Logical dynamic games: Models, equilibria,
  and potentials,'' \emph{IEEE Transactions on Automatic Control}, 2024.

\bibitem{huang2024distributed}
Y.~Huang, Z.~Meng, and J.~Sun, ``Distributed {N}ash equilibrium seeking for
  multicluster aggregative game of {E}uler--{L}agrange systems with coupled
  constraints,'' \emph{IEEE Transactions on Cybernetics}, pp. 1--12, 2024.

\bibitem{deng2022nash}
Z.~Deng and Y.~Liu, ``Nash equilibrium seeking algorithm design for distributed
  nonsmooth multicluster games over weight-balanced digraphs,'' \emph{IEEE
  Transactions on Neural Networks and Learning Systems}, vol.~34, no.~12, pp.
  10\,802--10\,811, 2022.

\bibitem{Wu2020}
C.~Wu, W.~Gu, R.~Bo, H.~MehdipourPicha, P.~Jiang, Z.~Wu, S.~Lu, and S.~Yao,
  ``Energy trading and generalized {N}ash equilibrium in combined heat and
  power market,'' \emph{IEEE Transactions on Power Systems}, vol.~35, no.~5,
  pp. 3378--3387, 2020.

\bibitem{Zhou2005}
J.~Zhou, W.~H. Lam, and B.~G. Heydecker, ``The generalized {N}ash equilibrium
  model for oligopolistic transit market with elastic demand,''
  \emph{Transportation Research Part B: Methodological}, vol.~39, no.~6, pp.
  519--544, 2005.

\bibitem{Bianchi2021}
M.~Bianchi and S.~Grammatico, ``Continuous-time fully distributed generalized
  {N}ash equilibrium seeking for multi-integrator agents,'' \emph{Automatica},
  vol. 129, p. 109660, 2021.

\bibitem{deng2018distributed}
Z.~Deng and X.~Nian, ``Distributed generalized {N}ash equilibrium seeking
  algorithm design for aggregative games over weight-balanced digraphs,''
  \emph{IEEE Transactions on Neural Networks and Learning Systems}, vol.~30,
  no.~3, pp. 695--706, 2018.

\bibitem{meng2022attack}
Q.~Meng, X.~Nian, Y.~Chen, and Z.~Chen, ``Attack-resilient distributed {N}ash
  equilibrium seeking of uncertain multiagent systems over unreliable
  communication networks,'' \emph{IEEE Transactions on Neural Networks and
  Learning Systems}, 2022.

\bibitem{deng2021distributed}
Z.~Deng, ``Distributed generalized {N}ash equilibrium seeking algorithm for
  nonsmooth aggregative games,'' \emph{Automatica}, vol. 132, p. 109794, 2021.

\bibitem{franci2021training}
B.~Franci and S.~Grammatico, ``Training generative adversarial networks via
  stochastic {N}ash games,'' \emph{IEEE Transactions on Neural Networks and
  Learning Systems}, vol.~34, no.~3, pp. 1319--1328, 2021.

\bibitem{Wang2021}
Z.~Wang, F.~Liu, Z.~Ma, Y.~Chen, M.~Jia, W.~Wei, and Q.~Wu, ``Distributed
  generalized {N}ash equilibrium seeking for energy sharing games in
  prosumers,'' \emph{IEEE Transactions on Power Systems}, vol.~36, no.~5, pp.
  3973--3986, 2021.

\bibitem{wilson1971computing}
R.~Wilson, ``Computing equilibria of n-person games,'' \emph{SIAM Journal on
  Applied Mathematics}, vol.~21, no.~1, pp. 80--87, 1971.

\bibitem{xia2023collaborative}
Z.~Xia, Y.~Liu, W.~Yu, and J.~Wang, ``A collaborative neurodynamic optimization
  approach to distributed {N}ash-equilibrium seeking in multicluster games with
  nonconvex functions,'' \emph{IEEE Transactions on Cybernetics}, 2023.

\bibitem{chen2024approaching}
G.~Chen, G.~Xu, F.~He, Y.~Hong, L.~Rutkowski, and D.~Tao, ``Approaching the
  global {N}ash equilibrium of non-convex multi-player games,'' \emph{IEEE
  Transactions on Pattern Analysis and Machine Intelligence}, 2024.

\bibitem{chatterjee2009optimization}
B.~Chatterjee, ``An optimization formulation to compute {N}ash equilibrium in
  finite games,'' in \emph{2009 Proceeding of International Conference on
  Methods and Models in Computer Science (ICM2CS)}.\hskip 1em plus 0.5em minus
  0.4em\relax IEEE, 2009, pp. 1--5.

\bibitem{xia2005recurrent}
Y.~Xia and J.~Wang, ``A recurrent neural network for solving nonlinear convex
  programs subject to linear constraints,'' \emph{IEEE Transactions on Neural
  Networks}, vol.~16, no.~2, pp. 379--386, 2005.

\bibitem{liu2016collective}
Q.~Liu, S.~Yang, and J.~Wang, ``A collective neurodynamic approach to
  distributed constrained optimization,'' \emph{IEEE Transactions on Neural
  Networks and Learning Systems}, vol.~28, no.~8, pp. 1747--1758, 2016.

\bibitem{yan2016collective}
Z.~Yan, J.~Fan, and J.~Wang, ``A collective neurodynamic approach to
  constrained global optimization,'' \emph{IEEE Transactions on Neural Networks
  and Learning Systems}, vol.~28, no.~5, pp. 1206--1215, 2016.

\bibitem{li2019distributed}
X.~Li, L.~Xie, and Y.~Hong, ``Distributed continuous-time nonsmooth convex
  optimization with coupled inequality constraints,'' \emph{IEEE Transactions
  on Control of Network Systems}, vol.~7, no.~1, pp. 74--84, 2019.

\bibitem{jiang2022second}
X.~Jiang, S.~Qin, X.~Xue, and X.~Liu, ``A second-order accelerated neurodynamic
  approach for distributed convex optimization,'' \emph{Neural Networks}, vol.
  146, pp. 161--173, 2022.

\bibitem{wei2023distributed}
M.~Wei, W.~Yu, H.~Liu, and Q.~Xu, ``Distributed weakly convex optimization
  under random time-delay interference,'' \emph{IEEE Transactions on Network
  Science and Engineering}, vol.~11, no.~1, pp. 212--224, 2023.

\bibitem{aubin2009differential}
J.-P. Aubin, H.~Frankowska, J.-P. Aubin, and H.~Frankowska, \emph{Differential
  Inclusions}.\hskip 1em plus 0.5em minus 0.4em\relax Springer, 2009.

\end{thebibliography}

\end{document}